\pdfoutput=1
\documentclass[10pt,a4paper]{amsart}[2004/08/06]
\usepackage{microtype}
\usepackage{fixltx2e}[2006/03/24]
\providecommand*{\TextOrMath}[2]{\ifmmode#2\else#1\fi}
\usepackage[latin2]{inputenc}
\usepackage[T1]{fontenc}
\usepackage{lmodern}
\usepackage{textcomp}

\usepackage[mathic, showonlyrefs]{mathtools}
\usepackage{amssymb}

\usepackage{tikz}
\usetikzlibrary{positioning,matrix,chains,calc}

\usepackage{enumitem}
\setenumerate[1]{label=(\arabic*)}
\setenumerate[2]{label=(\alph*), ref=(\theenumi.\alph*)}
\newcommand*{\itemtitle}[1]{\textbf{#1:}}

\usepackage{deepestsection}

\usepackage[bookmarksnumbered,bookmarksopen,unicode]{hyperref}

\hypersetup{
  pdftitle={Geometry of splice-quotient singularities},
  pdfauthor={Gábor Braun},
  pdfkeywords={splice-quotient singularity, End Curve Theorem,
     divisorial filtration},
  pdfsubject={MSC2000 Primary 14F05,14J25;
    Secondary 14C17, 32S05, 32S50},
}

\numberwithin{equation}{deepestsection}

\newcommand{\mynewtheorem}[2]{%
  \newtheorem{#1}{#2}[deepestsection]%
  \expandafter\DeclareRobustCommand\expandafter{\csname#1autorefname\endcsname}{#2}%
  \expandafter\let
    \csname c@#1\expandafter\endcsname
    \expandafter=%
    \csname c@equation\endcsname}

\theoremstyle{plain}
\mynewtheorem{theorem}{Theorem}
\mynewtheorem{lemma}{Lemma}
\mynewtheorem{proposition}{Proposition}
\mynewtheorem{corollary}{Corollary}

\theoremstyle{definition}
\mynewtheorem{definition}{Definition}
\mynewtheorem{example}{Example}
\mynewtheorem{remark}{Remark}
\mynewtheorem{exercise}{Exercise}

\newcommand*{\ApplyFunctionSymbol}{}
\newcommand*{\ApplyFunction}[2]{\mathinner{{#1}\ApplyFunctionSymbol{%
    \left(#2\right)}}}
\newcommand*{\Pic}[1]{\ApplyFunction{\operatorname{Pic}}{#1}}

\newcommand*{\linebundle}{\mathcal{L}}
\newcommand{\VertexSet}[1]{\mathinner{#1}}%old def: {\mathcal{V}}(#1)
\newcommand*{\FiltrationLevel}[1]{\ApplyFunction{\operatorname{\mathcal{F}}}{#1}}

\newcommand*{\ClassMinCycle}[1]{D\sb{#1}}

\newcommand{\setC}{\mathbb{C}}
\newcommand{\setQ}{\mathbb{Q}}

\newcommand{\setZ}{\mathbb{Z}}

\newcommand*{\size}[1]{\left\lvert#1\right\rvert}
\newcommand*{\cyclegroup}{L}
\newcommand*{\effectivecycles}{\cyclegroup\sb{{} \geq 0}}
\newcommand*{\dualcyclegroup}{\cyclegroup'}
\newcommand*{\ClassOf}[1]{{\left[#1\right]}}
\newcommand*{\Chernclasssymbol}{\mathinner{c\sb{1}}}
\newcommand*{\Chernclass}[1]{\ApplyFunction{\Chernclasssymbol}{#1}}
\newcommand*{\degree}[1]{\delta\sb{#1}}
\newcommand*{\restrictto}[2]{{\left.#1\right|}\sb{#2}}
\newcommand*{\scalarproduct}[2]{\mathinner{\left( #1, #2 \right)}}
\providecommand{\coloneqq}{\mathrel{:=}}

\newcommand*{\twistedbundle}[2]{\mathinner{{#1}%
  \ApplyFunctionSymbol {\left(#2\right)}}}
\newcommand*{\relativebundle}[2]{\mathinner{{#1}\sb{#2}}}
\newcommand*{\bundlefromdivisorsymbol}{\mathcal{O}}
\newcommand*{\bundlefromdivisor}[2][]{\twistedbundle{%
    \relativebundle{\bundlefromdivisorsymbol}{#1}}{#2}}
\newcommand*{\functionring}[1]{\mathinner{{\bundlefromdivisorsymbol}\sb{#1}}}
\newcommand*{\zerocohom}[1]{\ApplyFunction{H\sp{0}}{#1}}
\newcommand*{\firstcohom}[1]{{\ApplyFunction{H\sp{1}}{#1}}}
\newcommand*{\secondhom}[1]{\ApplyFunction{H\sb{2}}{#1}}
\newcommand*{\secondcohom}[1]{\ApplyFunction{H\sp{2}}{#1}}
\newcommand*{\smartfrac}[2]{\mathchoice
              {\frac{#1}{#2}}%                 \displaystyle
  {\mathinner{\left.{#1}\middle/{#2}\right.}}% \textstyle
  {\mathinner{\left.{#1}\middle/{#2}\right.}}% \scriptstyle
  {\mathinner{\left.{#1}\middle/{#2}\right.}}% \scriptscriptstyle
}
\newcommand*{\factorsections}[2]{%
  \smartfrac{\zerocohom{#1}}{\zerocohom{#2}}}
\newcommand*{\factorsectionsmodulo}[2]{\factorsections{#1}{%
    \twistedbundle{#1}{#2}}}
\newcommand*{\factorsectionstwotwists}[3]{\factorsections%
  {\twistedbundle{#1}{#2}}%
  {\twistedbundle{#1}{#3}}}
\newcommand*{\factorsectionsfromdivisor}[2]{%
  \factorsections{\bundlefromdivisor{#1}}{\bundlefromdivisor{{#1}-{#2}}}}
\newenvironment{diagram}{\begin{center}}{\end{center}}

\begin{document}

\title{Geometry of splice-quotient singularities}

\date{}

\author{Gábor Braun}
\address{Rényi Institute of Mathematics\\
  Budapest\\
  Reáltanoda u. 13\textendash15\\
  1053\\
 Hungary}
\email{braung@renyi.hu}



\begin{abstract}
  We obtain a new important basic result
  on splice-quotient singularities
  in an elegant combinatorial-geometric way:
  every level of
  the divisorial filtration of
  the ring of functions is generated by monomials of
  the defining coordinate functions.
  The elegant way is the language of of line bundles
  based on Okuma's description of
  the function ring of the universal abelian cover.
  As an easy application,
  we obtain
  a new proof of the End Curve Theorem of Neumann and Wahl.
\end{abstract}

\maketitle

\section{Introduction}
\label{sec:introduction}

\subsection{Results and applications}
\label{sec:results-applications}

In this article we consider isolated complex normal surface singularities
whose link is a rational homology sphere,
ie the resolution graph is a tree of projective lines.
A nice subclass of these are formed by
splice-quotient singularities,
which are generalizations of 
weighted-homogeneous singularities
by Neumann and Wahl
\cite[Definition 7.1 \& Theorem 7.2]{NWuj2}.
This class includes
rational singularities and
minimally elliptic singularities by \cite[Theorem 5.1.]{Ouac-c}.
For splice-quotient singularities,
the resolution graph determines
the leading terms of the equations for the universal abelian cover.
Therefore one may expect that the resolution graph determines
many analytical properties of these singularities.
Our aim is to provide a basis for such results.
On the other hand,
we point out in \autoref{sec:count-hilb-funct}
that the Hilbert\textendash Samuel function is \emph{not} determined by
the resolution graph alone.

For every singularity,
vanishing order on the exceptional curves produces
the \emph{divisorial} filtration
on the ring of functions of the universal abelian cover.
Our main result (\autoref{th:06}) is that
the divisorial filtration for splice-quotient singularities
is algebraic ie coming from the equations:
every level is generated by monomials.
As a special case,
this contains Okuma's result \cite[(3.3)]{Opg} about
the filtration defined by the vanishing order of only one
exceptional curve.

This has many interesting consequences.

First of all,
it leads to an easy proof of the
geometric characterization of splice-quotient singularities,
namely, the \hyperref[th:3]{End Curve Theorem~\ref*{th:3}}
\cite[page~2]{neumann-2008}.

Second,
Némethi \cite[Proposition~3.1.4.(1) with Theorem~4.1.1.]{nemethi08.cohomology_splice-quotient}
has used our result to
explicitly compute the dimension of factors of
different levels of the divisorial filtration
and to determine the multiplicity of the singularity among others.
This also gives formulas for the dimension of the cohomology of
an important class of line bundles, which we shall call
\emph{natural} line bundles.
These complement 
the additivity of the geometric genus by Okuma \cite[Theorem~4.5]{Opg}
and the Seiberg\textendash Witten Invariant Conjecture for splice-quotients 
(special case: \cite[Theorem on page~1]{NO1}, general case: \cite[Corollary 2.2.4]{BN}).

\subsection{Method of proof}
\label{sec:method-proof}

We employ a purely geometric approach to splice-quotient singularities,
exploiting only their geometric characterization:
the \hyperref[def:1]{End Curve Condition (\autoref*{def:1})}.
We shall consequently use the language of
line bundles
demonstrating that it is well-suited to
the study of splice-quotient singularities.
Therefore we present the definitions from the literature
rephrased in this language.
(The only exception is
the counterexample in \autoref{sec:count-hilb-funct}, which
involves more deliciate features.)

Thanks to Okuma's description of
the function ring of the universal abelian cover
via natural line bundles,
we can also rephrase our main result on the divisorial filtration
in the language of line bundles:
the essential part is that
factors of sections of natural line bundles are generated by
monomials of the defining coordinate functions.
(\autoref{lem:4}\ref{item:10}).

To ease the inductive proof,
the lemma contains another result
claiming that natural line bundles have many global sections
compared to local ones.
It is interesting that
this second result has become the core of Némethi's cohomology formulas
\cite[Theorem~4.1.1. \& Corollary~4.1.3.]{nemethi08.cohomology_splice-quotient}.

\subsection{Acknowledgements}
\label{sec:acknowledgements}

To a great extent,
this work builds on the article \cite{NWuj2} of Neumann and Wahl,
the articles \cite{Ouac-c,MR2078895,Opg} of Okuma
and a long series of discussion with Némethi
on isolated surface singularities.

\section{Notation and setup}
\label{sec:notation-setup}

\subsection{Resolutions}
\label{sec:resolutions}

Let \((X,o)\) be a complex normal surface singularity whose link
is a rational homology sphere.  Let \(\pi\colon \widetilde{X} \to X\)
be a good resolution with dual graph \(\Gamma\).  Recall that the link
being a rational homology sphere means that \(\Gamma\)
is a tree and all the irreducible exceptional divisors have genus
\(0\).

We will use the same notation for a graph and its set of vertices,
so \(v \in \VertexSet{\Gamma}\) means \(v\) is a vertex of \(\Gamma\).

Let \(\cyclegroup \coloneqq \secondhom{\widetilde{X}, \setZ}\).  It is freely generated by the
classes of the irreducible exceptional curves \({\{E_v\}}_{v \in \VertexSet{\Gamma}}\),
hence \(\cyclegroup\) is also the group of integral divisors
supported on the exceptional curves.
Let \(\dualcyclegroup\) denote \(\secondcohom{\widetilde{X}, \setZ}\).
Via Poincaré duality,
\(\dualcyclegroup\) is the dual of \(\cyclegroup\),
so it is freely generated by the duals \(E_v^*\) of the \(E_v\),
determined by \(\scalarproduct{E_v^*}{E_v} = - 1\) and
\(\scalarproduct{E_v^*}{E_w} = 0\) for \(v \neq w\).

The intersection form \(\scalarproduct{-}{-}\) on \(\cyclegroup\) provides an
embedding \(\cyclegroup \rightarrowtail \dualcyclegroup\) with factor the first homology group \(H\)
of the link \(\partial \widetilde{X}\).
We denote by \(\ClassOf{l'}\) the image of \(l' \in \dualcyclegroup\) in \(H\).
The intersection form \(\scalarproduct{-}{-}\) extends to \(\dualcyclegroup\).

The above embedding \(\cyclegroup \rightarrowtail \dualcyclegroup\)
induces an isomorphism \(\cyclegroup \otimes \setQ \simeq \dualcyclegroup \otimes \setQ\) and hence
realizes \(\dualcyclegroup\) as a subgroup of \(\cyclegroup \otimes \setQ\) .
The \(\setQ\) vector space \(\cyclegroup \otimes \setQ\) has a natural partial ordering:
the elements greater than or equal to \(0\) are the elements
with non-negative coefficients in the base of the \(E_v\).
By restriction to subgroups,
this provides a partial ordering on \(\cyclegroup\) and \(\dualcyclegroup\).
Let \(\effectivecycles\)
denote the semigroup of elements greater than or equal to \(0\)
of \(\cyclegroup\).
In other words,
\(\effectivecycles\) is the set of \emph{effective} divisors
supported on the exceptional curves.

Let
\(\theta\colon H \to \widehat{H}\)
be the isomorphism \(\ClassOf{l'} \mapsto e^{2 \pi i \scalarproduct{l'}{\cdot}}\)
of \(H\) with its Pontrjagin dual \(\widehat{H}\).

\subsection{Natural line bundles and the divisorial filtration}
\label{sec:natural-bundles-divisorial-filtration}

\subsubsection{Natural line bundles}
\label{sec:natural-line-bundles}

It is well-known, see \cite[3.1]{Line} or \cite[Lemma~2.2]{MR2078895},
that the first Chern class mapping from the Picard group
to the second cohomology group is onto
and has a group section \(\bundlefromdivisorsymbol\) whose image
contains the line bundles associated to
divisors supported on the exceptional curves.
The image is actually unique.
(Here is an easy, totally algebraic proof of the existence and uniqeness of the section.
The claim is equivalent to the unique splitting of the induced short exact sequence
\(
\firstcohom{\widetilde{X}, {\bundlefromdivisorsymbol}_{\widetilde{X}}}
\rightarrowtail
\smartfrac{\Pic{\widetilde{X}}}{\cyclegroup}
\twoheadrightarrow
\smartfrac{\dualcyclegroup}{\cyclegroup}
\).
This sequence really splits uniquely as it is an extension of
a torsion-group by a torsion-free divisible group.)
\begin{diagram}
\begin{tikzpicture}
  \begin{scope}[start chain,every node/.style={on chain},
    every join/.style=->]
    \node [join] (start0) {\(0\)};
    \node [join] (discrepancy)
    {\(\firstcohom{\widetilde{X},
        {\bundlefromdivisorsymbol}_{\widetilde{X}}}\)};
    \node [join] (Picard group) {\(\Pic{\widetilde{X}}\)};
    \node  (cohomology2) {\(\dualcyclegroup\)};
    \node [join] (end0) {\(0\)};
  \end{scope}
  \draw [->]
    ($ (Picard group.east) + (0,.5ex) $) to node[auto] {\(\Chernclasssymbol\)}
    ($ (cohomology2.west) + (0,.5ex) $);
  \draw [->,dashed]
     ($ (cohomology2.west) - (0,.5ex) $) to node[auto]
        {\(\bundlefromdivisorsymbol\)}
     ($ (Picard group.east) - (0,.5ex) $);

  \node (cycles) [above=of cohomology2] {\(\cyclegroup\)}
  edge [->] (cohomology2)
  edge [->,dashed] (Picard group);
\end{tikzpicture}
\end{diagram}
We shall call the line bundles of this subgroup
as \emph{natural} line bundles.
They are important
eg in the study of the universal abelian cover
but have so far no name in the literature.

\subsubsection{Eigen-decomposition of the universal abelian cover}
\label{sec:eigen-decomp-univ}

%% universal abelian cover: filtration & natural line bundles
Natural line bundles appear in the eigen-decomposition of the function ring of
the universal abelian cover.

Let \(c\colon (Y,o)\to (X,o)\) be the universal abelian cover of \((X,o)\).
Let \(\pi_Y \colon \widetilde{Y}\to Y\) be the (normalized) pullback of \(\pi\) by \(c\),
and let
\(\widetilde{c}\colon \widetilde{Y} \to \widetilde{X}\) be the morphism
covering \(c\).
\begin{diagram}
\begin{tikzpicture}
  %%Spaces
  \matrix [matrix of math nodes, column sep=2em, row sep=2em]{
    |(rescover)| \widetilde{Y} & |(res)| \widetilde{X} \\
    |(cover)| (Y,o)            & |(sing)| (X,o)        \\
  };
    %% Arrows
    \draw [->,auto] (res)      to node {\(\pi\)}            (sing);
    \draw [->,auto] (cover)    to node {\(c\)}             (sing);
    \draw [->,auto] (rescover) to node {\(\pi_Y\)}           (cover);
    \draw [->,auto] (rescover) to node {\(\widetilde{c}\)} (res);
\end{tikzpicture}
\end{diagram}
Then the action of \(H\) on \((Y,o)\) lifts to \(\widetilde{Y}\), and
one has the following eigenspace decomposition
(\cite[Lemma~3.5]{Opg}, \cite[(3.7)]{Line}).
The eigenspaces are parametrized by a subset of \(\dualcyclegroup\)
\begin{equation}
  \label{eq:1}
  Q \coloneqq \left\{ \sum l'_i E_i \in \dualcyclegroup \middle|
    0 \leq l_i' < 1 \right\}.
\end{equation}
The actual decomposition is
\begin{equation}
  \label{eq:2}
  {\widetilde{c}}_* \ApplyFunctionSymbol \functionring{\widetilde{Y}} =
  \bigoplus_{l'\in Q} \bundlefromdivisor{-l'}.
\end{equation}
Here and below a line bundle \(\linebundle\) on the right-hand side is
the \(\ApplyFunction{\theta}{- \ClassOf{\Chernclass{\linebundle}}}\)-eigenspace of the left-hand side.
More generally, one has
\begin{equation}
  \label{eq:21}
  {\widetilde{c}}_* \ApplyFunctionSymbol {\widetilde{c}}^*
  \ApplyFunctionSymbol \linebundle =
  \bigoplus_{l'\in Q} \linebundle \otimes \bundlefromdivisor{-l'}
  \quad\text{for every line bundle \(\linebundle\) over \(\widetilde{X}\)}.
\end{equation}

The class \(Q\) is a representative set of \(\smartfrac{\dualcyclegroup}{\cyclegroup}\).
For any \(l' \in \dualcyclegroup\),
let \(\ClassMinCycle{l'} \in Q\) be the unique
element of \(Q\) representing \(l' + \cyclegroup\).
Thus all the eigenspaces of the ring
\({{\widetilde{c}}_* \functionring{\widetilde{Y}}}\)
have the form \(\bundlefromdivisor{- \ClassMinCycle{l'}}\).
The multiplication in the ring is given by inclusion of line bundles
\(\bundlefromdivisor{-l'_1} \otimes \bundlefromdivisor{-l'_2} \simeq
\bundlefromdivisor{-l'_1 -l'_2} \hookrightarrow
\bundlefromdivisor{-\ClassMinCycle{l'_1 + l'_2}}\).

\begin{remark}
  \label{rem:3}
  We have deliberately defined
  the line bundles \(\bundlefromdivisor{-l'}\)
  and the homomorphisms
  appearing in the eigen-decomposition and multiplication
  only up to isomorphism,
  ie up to a multiplication by an invertible function.
  This will suffice for our purposes,
  since we will be interested only in divisors of sections.

  It is possible to define
  these exactly,
   but this involves only arbitrary choices
   and no new real relation.
\end{remark}

\begin{remark}
  \label{rem:2}
  The action of \(H\) on
  the universal abelian cover differs by a sign
  between the sources \cite{Opg} and \cite{NWuj2}.
  As we make more direct use of formulas from the former,
  we have adopted its sign convention.
\end{remark}

\subsubsection{Divisorial filtration}
\label{sec:divis-filtr}

For a resolution \(\pi\) of an isolated local surface singularity,
the \emph{divisorial filtration}
of the function ring \(\functionring{Y,o}\)
of the universal abelian cover
is defined as in \cite[(4.1.1)]{CDGb}):
the levels of the filtration are indexed by
\(l' \in \dualcyclegroup\) with \(l' \geq 0\) and
the level at \(l'\) is
\begin{equation}\label{eq:3}
  \begin{split}
    \FiltrationLevel{l'} \coloneqq
    \{ f \in \functionring{Y,o} \mid
    \ApplyFunction{\operatorname{div}}{f \circ \pi_Y} \geq \ApplyFunction{{\widetilde{c}}^*}{l'} \} &=
    \zerocohom{\widetilde{Y}, \bundlefromdivisor{- \ApplyFunction{{\widetilde{c}}^*}{l'}}} \\
    &\simeq \bigoplus_{l \in Q} \zerocohom{\widetilde{X}, \bundlefromdivisor{- l' - l}}.
  \end{split}
\end{equation}
Recall eg from \cite[(3.3)]{Line} that
for any \(l'\in \dualcyclegroup\), the pullback \(\ApplyFunction{{\widetilde{c}}^*}{l'}\) is
an \emph{integral cycle},
and hence is uniquely represented by a divisor
supported on \(\ApplyFunction{\pi_Y^{-1}}{o}\).
The notation \(\ApplyFunction{{\widetilde{c}}^*}{l'}\) means this divisor.

We see that for \(l'_1 \geq l'_2\),
the eigenspaces of the factor \(\smartfrac{\FiltrationLevel{l'_1}}{\FiltrationLevel{l'_2}}\)
have the form
\(\factorsectionsmodulo{\linebundle}{-l}\)
for some natural line bundle \(\linebundle\) and
effective divisor \(l \in \effectivecycles\).

\subsection{End curves and end curve sections}
\label{sec:end-curves-sect}

We define the fundamental geometric tool of splice-quotient singularities:
end curves
and their counterpart in the language of line bundles:
end curve sections.
\begin{definition}[End Curve Condition]
  \label{def:1}
  Let a good resolution of
  an isolated normal surface singularity
  be given.
  Let us consider an irreducible curve on it
  which intersects exactly one of the exceptional curves,
  and which intersects the exceptional curve transversally and
  in exactly one point (these are called transversal cuts).
  The curve is an \emph{end curve} of
  the exceptional curve it intersects
  if it is a divisor of a natural line bundle.
  An \emph{end curve section} of an end curve
  is a section of a natural line bundle
  having the end curve as a divisor.
  (The natural line bundle is obviously unique up to isomorphism.)

  The \emph{End Curve Condition} for the resolution
  is that every exceptional curve
  is intersected by exceptional curves and end curves
  in \emph{at least two} different points.
\end{definition}

For end curves of a vertex \(v\),
their end curve sections are sections of \(\bundlefromdivisor{- E_v^*}\).
Using \eqref{eq:2},
via the inclusion
\(\bundlefromdivisor{- E_v^*} \subseteq
\bundlefromdivisor{- \ClassMinCycle{E_v^*}} \subseteq \functionring{Y,o}\),
we shall regard the end curve sections as eigen-vectors of the function ring of
the universal abelian cover \((Y,o)\).

Recall that
every splice-quotient singularity satisfies
the \hyperref[def:1]{End Curve Condition}
by design (\cite[Theorem~7.2(6)]{NWuj2}).
We recall the construction in \hyperref[sec:defin-splice-quot]{the next section}.

\begin{remark}
  \label{rem:4}
  The \hyperref[def:1]{End Curve Condition} is usually formulated as
  every exceptional curve of degree \(1\) has
  at least one end curve intersecting it in a point
  different from that of its neighbour exceptional curve.
  Our formulation is equivalent to this one,
  and is better suited to the spirit of the present article.
  Note that graphs with only \(1\) vertex are rational,
  and hence they have the required end curves.
\end{remark}

\begin{remark}
  \label{rem:1}
  Instead of being a divisor of a natural line bundle,
  the original definition for an end curve required
  the equivalent condition stating
  the existence of a non-constant function
  whose divisor is supported on the exceptional divisor and the curve.
  Such a function is called an \emph{end curve function} for the curve.

  So far in the literature,
  end curves were defined only for ends, ie exceptional curves
  having degree one in the resolution graph.
  Since we shall use end curves of other exceptional curves also
  (in the proof of Proposition~\ref{prop:end-to-monom}),
  we consider less confusing
  to extend the notion of end curve than to
  introduce a new name.
  End curves at other exceptional curves appear naturally anyway,
  for example when passing to the restriction to a subgraph,
  see \autoref{lem:1}\ref{item:9}.
\end{remark}

\begin{remark}\label{rem:31}
  The \hyperref[def:1]{End Curve Condition} is really
  a property of the resolution.
  There are singularities which have a resolution
  satisfying the \hyperref[def:1]{End Curve Condition}
  and also a resolution
  not satisfying it.
  The two resolutions can even have the same graph.

  However, blowdowns preserve the \hyperref[def:1]{End Curve Condition},
  so if a resolution of a singularity
  satisfies the End Curve Condition,
  then so does the minimal resolution.
\end{remark}

\section{Definition of splice-quotient singularities}
\label{sec:defin-splice-quot}

Splice-quotient singularities are defined as
the result of the following construction.
The spirit of the construction is that
we want to construct a singularity
with a given resolution graph together 
with a given collection of end curves 
demonstrating the \hyperref[def:1]{End Curve Condition},
using end curve sections as
coordinates for the universal abelian cover.

The construction is recalled from \cite[Definition~7.1]{NWuj2}
with some reformulations from \cite[Section~2.2]{Opg}
and a slight generalization to encompass our
extended notion of end curves.

For the convenience of the reader,
we shall include the geometric meaning of every notion
in parenthesis below.
Of course,
these parenthesised expressions are not part of the formal definitions.

%%% Data:
The construction starts with the following data:
a tree \(\Gamma\) with at least \(3\) vertices
(the tree of all exceptional curves and some end curves)
together with integers at vertices of degree at least \(2\),
such that the subtree of vertices of degree at least \(2\)
is negative definite.
(Vertices with degree \(1\) are end curves,
the other vertices are exceptional curves.)

Let \(\cyclegroup \coloneqq \bigoplus_{u\colon \degree{u}\geq2} E_u\) be the group
generated by the vertices of degree at least \(2\) together with
the negative definite symmetric bilinear map \(\scalarproduct{-}{-}\)
induced by the numbers at these vertices.
Let
\(\dualcyclegroup\) be the dual of \(\cyclegroup\),
which is freely generated by the duals \(E_v^*\) of the \(E_v\),
defined via
\(\scalarproduct{E_v^*}{E_v} = - 1\) and
\(\scalarproduct{E_v^*}{E_u} = 0\) for \(v \neq u\).
(The groups \(\dualcyclegroup\) and \(\cyclegroup\)
are the second homology group and second cohomology group of the
resolution, respectively.
The elements \(E_v\) and \(E_v^*\) are the canonical generators,
as introduced in \autoref{sec:resolutions}.)
Let \(H \coloneqq \smartfrac{\dualcyclegroup}{\cyclegroup}\)
(the first homology group of the link).

For every vertex \(w\) of degree \(1\),
let \(z_w\) be a variable
(an end curve section of \(w\)).
Furthermore,
let \(E_w^* \coloneqq E_{u}^*\),
where \(u\) is the unique neighbour of \(w\).

In \autoref{sec:degr-monom-cond},
we make some definitions,
which will be used in the construction
in \autoref{sec:constr-splice-quot}.

\subsection{Degrees and Monomial Condition}
\label{sec:degr-monom-cond}

We introduce a rudimentary form of
notions related to divisors of functions, ie power series in the \(z_w\).

\begin{definition}\label{def:5}
  For every vertex \(v\),
  the \emph{\(v\)-degree} of
  a monomial \(\prod_{i} z_i^{\alpha_i}\) is
  \begin{equation}\label{eq:25}
    - \sum_i \alpha_i \scalarproduct{E_v^*}{E_i^*}.
  \end{equation}
  The \emph{\(v\)-degree} (vanishing order on \(E_v\)) of a power series is
  the minimum of the \(v\)-degrees of
  the monomials having non-zero coefficient.
\end{definition}

\begin{remark}\label{rem:deg}
  In general, the \(v\)-degree is a non-negative rational number.
  In the literature, there is an additional constant factor in the definition
  to make it an integer.
  We prefer our choice, because it will make the \(v\)-degree of a function
  to be its vanishing order on \(E_v\).
\end{remark}

\begin{remark}
  In the literature,
  the term \(v\)-degree, \(v\)-order and \(v\)-weight are used
  in the same meaning.
  For clarity,
  we use only \(v\)-degree.
\end{remark}

\begin{definition}\label{def:4}
  For any tree \(\Gamma\)
  we define the following.
  \begin{enumerate}
  \item\label{item:20}
    The \emph{branches} of a vertex \(v\) are the components of \(\Gamma \setminus v\).
    The \emph{variables} of a branch are
    the variables \(z_w\) of the vertices \(w\) of the
    component
    having degree \(1\) in \(\Gamma\).
  \item\label{item:21} \itemtitle{Monomial Condition
      \cite[Condition~3.3]{Ouac-c}, \cite[Definition~2.4]{Opg}}
    The \emph{Monomial Condition} for a branch of a vertex \(v\) is
    the existence of a non-negative integer \(\alpha_w\) 
    for every vertex of the branch
    such that
    \begin{equation}
      \label{eq:24}
      \sum_{w\colon \degree{w}=1} \alpha_w E_w^* - \sum_{w\colon \degree{w} \geq 2} \alpha_w E_w = E_v^*,
    \end{equation}
    where \(\delta_w\) is the degree of \(w\) in \(\Gamma\).
    For such \(\alpha_w\),
    the monomial \(\prod_{w\colon \degree{w}=1} z_w^{\alpha_w}\) is
    \emph{admissible} for the branch.
    (In other words,
    \(\prod_{w\colon \degree{w}=1} z_w^{\alpha_w}\) is admissible if
    it is a holomorphic section of
    \(\bundlefromdivisor{- E_v^*}\)
    with divisor supported on the branch
    via the inclusions
    \(\zerocohom{\bundlefromdivisor{- E_v^*}} \subseteq 
    \zerocohom{\bundlefromdivisor{- \ClassMinCycle{E_v^*}}} \subseteq 
    \functionring{Y,o}\).
    The divisor of
    \(\prod_{w\colon \degree{w}=1} z_w^{\alpha_w}\) in
    \(\bundlefromdivisor{E_v^*}\)
    is \(\sum_{w\colon \degree{w}=1} \alpha_w w - \sum_{w\colon \degree{w} \geq 2} \alpha_w E_w\)
    where in the first sum \(w\) also stands for the end curve it represents.)

    The \emph{Monomial Condition} for a negative definite tree is that
    for every vertex having at least \(3\) branches,
    all of its branches
    satisfy the Monomial Condition.
  \end{enumerate}
\end{definition}

The \(v\)-degree of all admissible monomials of
the branches of \(v\) is
\(- \scalarproduct{E_v^*}{E_v^*}\)
as easily seen from~\eqref{eq:24}.

\subsection{Construction of splice-quotient singularities}
\label{sec:constr-splice-quot}

We are ready to recall the construction of splice-quotient singularities.
For the reader's convenience,
we recall the given data and some notation
from the beginning of \hyperref[sec:defin-splice-quot]{this section}.
\begin{definition}[{Splice diagram equations \cite[Definition~7.1]{NWuj2}}]
  \label{def:6}
  %%%Construction:
  %%%%Data and conditions:
  Let \(\Gamma\) be a tree with at least \(3\) vertices
  such that the subtree of vertices of degree at least \(2\)
  is negative definite.
  Let \(\Gamma\) satisfy the \hyperref[item:21]{Monomial Condition}.
  We shall define equations in the ring of convergent power series in the
  variables \(z_w\) using the notations introduced in this section.

  %%%%Group action:
  First, we define the action of the group \(H\) of \(\Gamma\) on the ring via
  \begin{equation}\label{eq:14}
    \ClassOf{E_v^*} \cdot z_w \coloneqq e^{- 2 \pi i \cdot \scalarproduct{E_v^*}{E_w^*}} z_w,
  \end{equation}
  ie \(z_w\) is a \(\ApplyFunction{\theta}{\ClassOf{E_w^*}}\)-eigenvector.

  %%%%Equations:
  We make some arbitrary choices.
  We select an admissible monomial \(M_{v,C}\)
  for every branch \(C\) of every vertex \(v\) with at least \(3\) branches.
  We select complex numbers \(a_{v,i,C}\) for \(1 \leq i \leq \degree{v} - 2\) such that
  for every \(v\),
  all the maximal minors of the matrix \((a_{v,i,C})\) have full rank
  (ie are non-degenerate).
  Finally, we choose convergent power series \(H_{v,i}\)
  for \(1 \leq i \leq \degree{v} - 2\), which are
  \(\ApplyFunction{\theta}{\ClassOf{E_v^*}}\)-eigenvectors of the \(H\)-action and
  have \(v\)-degree greater than \(- \scalarproduct{E_v^*}{E_v^*}\),
  ie greater than the \(v\)-degree of the \(M_{v,C}\).
  The \emph{splice diagram equations} are the equations
  for every vertex \(v\) with at least \(3\) branches:
  \begin{equation}
    \label{eq:9}
    \sum_C a_{v,i,C} M_{v,C} + H_{v,i} = 0, \quad i = 1, \dotsc, \degree{v} - 2,
  \end{equation}
  where \(C\) runs over the branches of \(v\).

  %%%%Factor:
  The splice diagram equations define a singularity with an \(H\)-action.
  The factor of the singularity by the \(H\)-action is
  the result of the construction.
  Singularities arising this way are called
  \emph{splice-quotient} singularities.
\end{definition}

%%% Properties: isolated, universal abelian cover, resolution graph
\begin{theorem}[{\cite[Theorem~7.2]{NWuj2}}]
  \label{th:6}
  Let \(\Gamma\) be a tree with at least \(3\) vertices
  such that the subtree of vertices of degree at least \(2\)
  is negative definite.
  Let \(\Gamma\) satisfy the \hyperref[item:21]{Monomial Condition}.

  Then the splice diagram equations of \(\Gamma\)
  define an isolated complete surface singularity
  with \(H\) acting freely on it outside the origin.
  It is the universal abelian cover of its factor by \(H\), ie of the
  splice-quotient singularity resulting from the construction.
  The factor is an isolated surface singularity with
  a good resolution whose resolution graph is
  the subtree of \(\Gamma\) consisting of
  vertices of degree at least \(2\).
  Every \(1\)-degree vertex of \(\Gamma\) is an end curve
  whose variable is an end curve section of it.
\end{theorem}

\begin{remark}
  Actually, \cite[Theorem~7.2]{NWuj2} claims the above only for
  quasi-minimal resolution graphs
  (a technical modification of minimal resolution graphs)
  with one end curve 
  at every vertex of degree \(1\).
  However,
  the proof can be easily extended to this general case.
\end{remark}

%% Equisingularity
The weak equisingularity type of
the constructed splice-quotient singularity
depends only on the graph.
Independence on the choice of the \(H_{v,i}\) and  \(a_{v,i,C}\)
are easy, see eg \cite[Theorem~1.1]{Ouac-c}.
Independence on the choice of admissible monomials
is stated in \cite[Theorem~10.1]{NWuj2}.

\section{Main results}
\label{sec:geom-prop-splice}

We formulate our main results on
splice-quotient singularities.

\begin{theorem}\label{th:06}
  For splice-quotient singularities,
  every level of
  the  divisorial filtration
  is generated by monomials of the defining coordinate functions
  (the \(z_w\) in \autoref{def:6}).
\end{theorem}

This has many applications,
for example the \emph{End Curve Theorem},
which states that every singularity
satisfying the End Curve Condition is
splice-quotient.

\begin{theorem}[{End Curve Theorem \cite[page~2]{neumann-2008}}]
  \label{th:3}
  If a good resolution of
  an isolated normal surface singularity with
  rational homology sphere
  satisfies the \hyperref[def:1]{End Curve Condition},
  then it is splice-quotient.
  Moreover, the resolution
  arises as a result of the splice-quotient construction
  in \autoref{def:6}
  with the coordinate functions being
  arbitrarily chosen end curve sections
  for an arbitrarily fixed set of end curves
  demonstrating the \hyperref[def:1]{End Curve Condition},
  ie every exceptional curve is intersected by at least two
  exceptional curves and end curves altogether with no three curves having a
  common point.
\end{theorem}
These results will be proved in \autoref{sec:proof-end-curve}.

Némethi has applied our \autoref{lem:4}\ref{item:11}.
for refining \autoref{th:06} to explicit formulas for
the cohomology of natural line bundles.
Without proof,
here we mention only a power series formula,
a direct generalization of \cite[Theorem~1]{MR2025300}
and its successor \cite[Theorem~5.1.5]{CDGb}
from rational singularities and minimal elliptic singularities
to splice-quotient singularities.
\begin{corollary}\label{cor:1}
  \cite[Theorem 4.1.1. in its (4.4.1) form]{nemethi08.cohomology_splice-quotient}
  For every resolution of a splice-quotient singularity
  satisfying the End Curve Condition,
  the following hold.
  \begin{equation}\label{eq:4}
    \sum_{k_v \geq 0}
    \sum_{I \subseteq \VertexSet{\Gamma}} {(-1)}^{\size{I} + 1}
    \dim \factorsectionsfromdivisor{- \sum_v k_v E_v^*}{\sum_{w \in I} E_w}
    \prod_{v \in \VertexSet{\Gamma}} x_v^{k_v} =
    \prod_{v \in \VertexSet{\Gamma}} {\left( 1 - x_v \right)}^{\degree{v} - 2}.
  \end{equation}
  Here \(\size{I}\) denotes the number of elements of the set \(I\).
\end{corollary}

\section{Numerically effective cycles}
\label{sec:numer-effect-cycl}

In this section,
we recall some combinatorial results regarding
the homological cycles of isolated surface singularities with
the link being a rational homology sphere.

The results are mostly about numerically effectiveness, which we recall.
\begin{definition}
  \label{def:9}
  A cycle \(l' \in \cyclegroup \otimes \setQ\) is \emph{numerically effective}
  if for every exceptional curve \(E_v\),
  we have \(\scalarproduct{E_v}{l'} \geq 0\).
  A line bundle is \emph{numerically effective}
  if its first Chern class is so.
\end{definition}

The basic property of numerically effective cycles is their relationship with
global sections with line bundles.
\begin{lemma}
  \label{lem:8}
  \cite[(4.2)(a) \& (c)]{Line}
  For every line bundle \(\linebundle\),
  there is a least effective cycle \(x \in \cyclegroup\)
  such that \(\twistedbundle{\linebundle}{-x}\) is numerically effective.
  Moreover,
  for every cycle \(0 \leq y \leq x\)
  the inclusion
  \(\zerocohom{\twistedbundle{\linebundle}{-y}} \hookrightarrow
  \zerocohom{\linebundle}\)
  is an isomorphism.

  For every vertex \(v\) with \(\scalarproduct{\Chernclass{\linebundle}}{E_v} < 0\), we have \(x \geq E_v\).
\end{lemma}

\begin{remark}
  Actually, the reference \cite[(4.2)(c)]{Line} is about first cohomology.
  Our reformulation for the zeroth cohomology can be obtained as follows.
  Via the cohomology exact sequence of the short exact sequence
  \(\twistedbundle{\linebundle}{-x} \to \linebundle \to
  \relativebundle{\linebundle}{x}\)
  the reference is equivalent to that the inclusion
  \(\zerocohom{\twistedbundle{\linebundle}{-x}} \hookrightarrow \zerocohom{\linebundle}\)
  is an isomorphism.  Applying this to \(\twistedbundle{\linebundle}{-y}\)
  instead of \(\linebundle\), we obtain that the inclusion
  \(\zerocohom{\twistedbundle{\linebundle}{-x}} \hookrightarrow
  \zerocohom{\twistedbundle{\linebundle}{-y}}\) is also an isomorphism.  Hence
  the claimed inclusion \(\zerocohom{\twistedbundle{\linebundle}{-y}} \hookrightarrow
  \zerocohom{\linebundle}\) is an isomorphism, as well.
\end{remark}

Another well-known result is about
the sparse distribution of numerically effective cycles.
\begin{lemma}
  \label{lem:9}
  For every \(l' \in \cyclegroup \otimes \setQ\),
  all but finitely many numerically effective \(l'_0 \in \dualcyclegroup\)
  satisfies
  \(l' \geq l'_0\).
  For example,
  all numerical effective cycles are less than or equal to \(0\).
\end{lemma}

The final result allows some simplification in the definition of
splice-quotient singularities.
\begin{lemma}
  \label{lem:2}
  (Cf~\cite[Lemma~3.2]{NWuj2})
  For every numerically effective cycle \(l' \in \cyclegroup \otimes \setQ\)
  and every vertex \(v\) of the resolution graph:
  \begin{equation*}
    l' \leq \frac{\scalarproduct{l'}{E_v^*}}{\scalarproduct{E_v^*}{E_v^*}} E_v^*.
  \end{equation*}
\end{lemma}

\begin{proof}
Let us express the difference
\(l' - \smartfrac{\scalarproduct{l'}{E_v^*}}{\scalarproduct{E_v^*}{E_v^*}} E_v^*\)
in the basis of the \(E_w\).
The coefficient of \(E_v\) is the scalar product with \(- E_v^*\),
which is \(0\).
Therefore the difference can be considered as a rational cycle over
the subgraph obtained by removing the vertex \(v\).
Over this subgraph, the scalar product with the \(E_w\) for \(w \neq v\)
remains the same as over the whole graph,
which is the same as the scalar product of \(l'\) with the \(E_w\).
Hence the difference
\(l' - \smartfrac{\scalarproduct{l'}{E_v^*}}{\scalarproduct{E_v^*}{E_v^*}} E_v^*\)
is numerically effective over the subgraph.
Thus, by \autoref{lem:9}, the difference is less than or equal to \(0\).
\end{proof}

\section{Restrictions}
\label{sec:restrictions}

In this subsection,
we show that
if a singularity satisfies the \hyperref[def:1]{End Curve Condition},
then so do the singularities determined by its subgraphs.
Moreover, the restrictions of natural line bundles to these singularities
are also natural.
These are mostly known results,
but we present the proof in the language of line bundles.

\begin{lemma}\label{lem:1}
  (Cf \cite[Proposition~2.16]{Opg}.)
  For a singularity satisfying the \hyperref[def:1]{End Curve Condition},
  the following hold.
  \begin{enumerate}
  \item\label{item:25}
    Every divisor supported on exceptional curves and end curves
    is a divisor of a natural line bundle.
  \item\label{item:9}
    Every subtree satisfies the \hyperref[def:1]{End Curve Condition}.
    In particular,
    the subtree has the following end curves:
    the exceptional curves and end curves of the original graph
    intersecting the subtree.
  \item\label{item:24}
    Restrictions of natural line bundles to subtrees are natural.
    Moreover,
    for
    every divisor on the subtree
    representing the first Chern class of the restriction of a natural bundle
    and which is supported on
    the subtree and
    the exceptional curves and end curves of the big graph
    intersecting the subtree,
    the following extensibility property holds.
    The divisor extends to a divisor of the original bundle
    supported on the exceptional curves and end curves of the original graph.
    In particular, the divisor
    is really a divisor of the restriction of the line bundle.
  \end{enumerate}
\end{lemma}

\begin{proof}
Since a natural line bundle is uniquely determined by its Chern class,
statement \ref{item:25} follows.

We show the extensibility property of divisors from~\ref{item:24}.
This is merely a combinatorial claim:
we have to extend a divisor to represent a given cohomology class.
This can be done step by step: each step adds a new vertex to the subtree
and extends the divisor to the larger subtree.
At every step all we have to do is find coefficients
for the
end curves and exceptional curves
intersecting the added vertex
which do not have a coefficient yet,
ie all of them but one.

The requirement to represent the given cohomology class is that
the intersection numbers of the extended divisor with
the exceptional curves of the extended subtree is
the same as that of the cohomology class.
This is automatic for all vertices except the newly added one.
For the newly added vertex,
the requirement is that the sum of the new coefficients be
a prescribed value, which can be fulfilled
as there is at least one undetermined coefficient
by the hypothesis that
the added vertex
(as every exceptional curve) is intersected by
at least two end curves and exceptional curves
of the original graph altogether.
Hence the extension of the divisor is possible.

It also follows that the original divisor is a divisor of the restriction.
This proves the extensibility property from~\ref{item:24}.

Restriction of natural line bundles form
a subgroup of the Picard group of the subtree,
which we shall call the \emph{restriction group}.
We shall show that this is the group of natural line bundles.

First,
the restriction group consists of the line bundles
associated to divisors supported
on the subtree and
the end curve candidates of \ref{item:9}
ie the exceptional curves and end curves of the original graph
intersecting the subtree.
In particular, 
the restriction group contains all line bundles associated to divisors
supported on the exceptional curves.

Obviously,
the \hyperref[def:1]{End Curve Condition} implies that
every exceptional curve of the subtree has at least two end curve candidates,
and thus the end curve candidates together with the exceptional curves
generate the second cohomology group.
Here we use that the second cohomology group \(\dualcyclegroup\)
is generated by the \(E_v^*\) where \(v\) runs through
all but one vertices of degree \(1\), see eg \cite[Proosition~5.1]{NWuj2}.
Hence the Chern class restricted to the restriction group is onto.

Second,
the kernel of the Chern class on the restriction group
consists of restrictions having Chern class \(0\).
By the part of~\ref{item:24} already proven,
these restrictions have divisor \(0\), ie
all of them are trivial.
Hence the Chern class is injective on the restriction group
finishing the proof that the restriction group is the group of natural
line bundles.

As a consequence,
the end curve candidates of \ref{item:9}
are really end curves
as they are divisors of natural line bundles.
\end{proof}

\section{Generators for global sections}
\label{sec:gener-glob-sect}

Now we are ready to formulate our main lemma.
It has two parts,
which are of different nature.
Nevertheless,
as we have already mentioned in the introduction,
we present them together since
the proof is a simultaneous induction of both statements.

\begin{lemma}[Main Lemma]\label{lem:4}
  Let \(\pi\colon \widetilde{X}\to X\) be a good resolution of a singularity
  satisfying the \hyperref[def:1]{End Curve Condition}.
  Let \(\Gamma\) denote the resolution graph.
  Let us have a fixed collection of end curves
  demonstrating the \hyperref[def:1]{End Curve Condition},
  ie every exceptional curve is intersected by
  at least two other exceptional curves and end curves altogether,
  and no three curves have a common point.
  Let \(\linebundle\) be a natural line bundle on \(\widetilde{X}\).

  Then the following hold.
  \begin{enumerate}
  \item\label{item:10}
    Let \(l \in \effectivecycles\) be an effective cycle
    on the exceptional curves.
    Then every subset of
    \(\factorsectionsmodulo{\linebundle}{-l}\)
    satisfying the following condition
    generates it as a vector space:
    \begin{quote}
      %%MISUSE of quote
      For every effective divisor of \(\linebundle\)
      supported on the exceptional curves and the fixed end curves
      which is
      not greater than or equal to \(l\),
      the set contains the class of a section of \(\linebundle\)
      with this divisor.
    \end{quote}
  \item\label{item:11}
    Let \(v\) be a vertex of degree \(1\) of the graph \(\Gamma\)
    such that
    \(\scalarproduct{\Chernclass{\linebundle}}{E_v} \geq 0\).
    Let \(w\) be the unique neighbour of \(v\).
    Then the restriction map
    \begin{equation}
      \label{eq:27}
      \factorsectionsmodulo{\linebundle}{- E_w}
      \longrightarrow
      \factorsectionsmodulo{\restrictto{\linebundle}{\Gamma \setminus v}}{- E_w}.
    \end{equation}
    is an isomorphism.
    Here \(\restrictto{\linebundle}{\Delta}\) denotes
    the restriction of \(\linebundle\) to the neighbourhood
    of the subgraph \(\Delta\).
  \end{enumerate}
\end{lemma}

\begin{proof}
We prove both statements by a simultaneous
induction on the number of vertices of the resolution graph.

First we prove \ref{item:11}
together with \ref{item:10} for the factors
\(\factorsectionsmodulo{\linebundle}{- E_w}\)
appearing in \ref{item:11}.
We shall use the notations of \ref{item:11}.
For brevity, let us call subsets satisfying the condition of
\ref{item:10} as \emph{generating subset candidates}.

Obviously, the restriction map is injective,
so we need to prove surjectivity
and that the generating subset candidates are
really generating subsets.
These two statements together are equivalent to
a single statement: the image of
generating subset candidates are
generating subsets.

As the reader may expect,
we use the induction hypothesis to show
that the images are generating subsets,
so we shall prove that the image of generating subset
candidates are again generating subset candidates.

Recall that being a generating subset candidate means that
certain divisors appear as divisors of some sections in the subset.
So we will prove the best we can hope for divisors,
which is sufficient for the claim of the previous paragraph
by \autoref{lem:1}\ref{item:25}:
every effective divisor of the restriction of \(\linebundle\)
supported on the exceptional curves and the fixed set of end curves
with the vanishing order on \(E_w\) being \(0\)
extends to an effective divisor of \(\linebundle\)
supported on the exceptional curves and the fixed set of end curves
also.
(This is not a special case of \autoref{lem:1}\ref{item:24}
as it does not guarantee an \emph{effective} extension to divisors.
But we shall extend its argument to the present case.)

Indeed, let \(k \geq 0\) be the vanishing order on \(E_v\) of such a divisor
of the restricted bundle.
To extend it,
we only need to choose
non-negative orders for the end curves of \(E_v\) such that
their sum is \(\scalarproduct{E_v}{\Chernclass{\linebundle}} - k E_v^2\),
which is non-negative since
\(\scalarproduct{E_v}{\Chernclass{\linebundle}} \geq 0\) and \(E_v^2 < 0\).
Hence the extension is possible.

Now we turn to statement \ref{item:10}.

We start with a special case
of a one-vertex graph.
Let \(v\) be the single vertex.
At the moment,
we prove \ref{item:10}
only for factors
\(\factorsectionsmodulo{\linebundle}{E_v}\)
with \(\Chernclass{\linebundle}\) numerically effective
ie \(\linebundle = -k E_v^*\) for some \(k \geq 0\).

We have an embedding
\begin{equation}
  \label{eq:19}
  \factorsectionsfromdivisor{- k E_v^*}{E_v}
  \rightarrowtail \zerocohom{\bundlefromdivisor[E_v]{- k E_v^*}}.
\end{equation}
The vector space on the right-hand side has dimension \(k+1\).

Let us have a generating subset candidate of the factor,
ie a subset
satisfying the condition of \ref{item:10}.
Let \(H_1\) and \(H_2\) be two of the fixed end curves.
Let us choose sections in the given set
 with divisors \(i H_1 + (k-i) H_2\) representing \(k E_v^*\) for
\(i=0, \dotsc, k\), which is possible by hypothesis.
Clearly,
the images of these \(k+1\) sections are linearly independent
in the right-hand side,
since they have different vanishing order on
the intersection of \(H_1\) and \(E_v\).
So they form a basis of the right-hand side,
hence they also form a basis of the left-hand side and
the inclusion is an isomorphism.

Now we turn to statement \ref{item:10} in general.

We construct recursively an increasing sequence
\(x_n \in \effectivecycles\), such that \(x_0 = 0\)
and the following are satisfied.
\begin{enumerate}[label=(\roman*)]
\item\label{item:12}
  The line bundle \(\twistedbundle{\linebundle}{- x_{2n+1}}\) is numerically effective.
\item\label{item:13}
  The statement of \ref{item:10} holds for
  the factor
  \(\factorsectionstwotwists{\linebundle}{- x_{n}}{- x_{n+1}}\).
\item\label{item:22}
  The sequence is strictly increasing at even steps:
  \(x_{2n-1} < x_{2n}\).
\end{enumerate}
Given \(x_{2n}\), we can find an \(x_{2n+1}\)
by \autoref{lem:8},
which even satisfies a condition stronger than \ref{item:13}:
\(\factorsectionstwotwists{\linebundle}{- x_{2n}}{- x_{2n+1}} = 0\).

Given \(x_{2n-1}\), we can always choose a suitable \(x_{2n}\)
by the special cases of \ref{item:10} already proved.
If the graph has only one vertex,
then we must have \(x_{2n-1} = - k E_v^*\) for some \(k \geq 0\)
and we can choose \(x_{2n} \coloneqq x_{2n-1} + E_v\)
where \(v\) is the single vertex of the graph.
If the graph has at least two vertices,
then we can choose
\(x_{2n} \coloneqq x_{2n-1} + E_w\)
where \(w\) is a neighbour of a vertex of degree \(1\).

The conditions imply that the \(\Chernclass{\twistedbundle{\linebundle}{- x_{2n+1}}}\)
are different numerically effective cycles, hence by \autoref{lem:9}
\(x_k \geq l\) for some odd \(k\).

Next we show by induction on \(n\) that
\(\factorsectionsmodulo{\linebundle}{- x_n}\)
satisfies \ref{item:10}.
This is obvious for \(n=0\).

The inductive step follows from the short exact sequence
\begin{equation*}
0 \longrightarrow
\factorsectionstwotwists{\linebundle}{- x_n}{- x_{n+1}}
\longrightarrow
\factorsectionsmodulo{\linebundle}{- x_{n+1}}
\longrightarrow
\factorsectionsmodulo{\linebundle}{- x_n}
\longrightarrow 0.
\end{equation*}

The right-hand side satisfies \ref{item:10} by the inductional hypothesis.
The left-hand side satisfies \ref{item:10}, too,
by~\ref{item:13}.
Since the left-hand side and the right-hand side of the exact sequence
satisfy \ref{item:10}, so does the middle term.

Finally, \(\factorsectionsmodulo{\linebundle}{-l}\) is a factor of
\(\factorsectionsmodulo{\linebundle}{- x_k}\), and hence it also
satisfies \ref{item:10}.
\end{proof}

\section{Geometric characterization of splice-quotient singularities}
\label{sec:proof-end-curve}

In this section,
we prove the \hyperref[th:3]{End Curve Theorem~\ref*{th:3}} and \autoref{th:06}
based on \autoref{lem:4}.
We will prove \autoref{th:06} for every singularity satisfying
the \hyperref[def:1]{End Curve Condition}
because we want to use it in the proof of the \hyperref[th:3]{End Curve Theorem~\ref*{th:3}}.

In \autoref{sec:divisors-monomials},
we determine the place of monomials of end curve sections
in the divisorial filtration,
as a preparation for proving \autoref{th:06}
in \autoref{sec:divis-filtr-gener-mon},
ie that
levels of the divisorial filtration are generated by monomials.
In particular, this will show that
the function ring of the universal abelian cover
consists of convergent power series in end curve sections,
which is a major step for the \hyperref[th:3]{End Curve Theorem~\ref*{th:3}}.

Our next step is to
derive the \hyperref[item:21]{Monomial Condition},
which is required by the splice-quotient construction,
from the \hyperref[def:1]{End Curve Condition}
in \autoref{sec:end-curve-monomial}.
Finally, we finish
the proof in
\autoref{sec:proving-end-curve}
by finding splice-diagram equations.

\subsection{Divisors of monomials}
\label{sec:divisors-monomials}

We describe the levels of the divisorial filtration of
the function ring of the universal cover,
which contain any given product of end curve sections. 
\begin{lemma}
  \label{lem:3}
  Let the \(z_i\) be end curve sections in the function ring of
  the universal abelian cover.
  Let \(E_i^*\) be the inverse of the second cohomology class
  represented by the end curve of \(z_i\).
  Then for every non-negative second cohomology class \(l'\)
  of a resolution of the singularity
  and every non-negative integers \(\alpha_i\):
  \begin{equation}
    \label{eq:22}
    \prod_{i} z_i^{\alpha_i} \in \FiltrationLevel{l'} \iff
    \sum_{i} \alpha_i E_i^* \geq l'.
  \end{equation}
\end{lemma}

\begin{proof}
Via the identifications of sections of natural line bundles with
eigen-functions on the universal abelian cover,
every end curve section \(z_i\) is a section of
\(\bundlefromdivisor{- E_i^*}\)
with divisor its end curve.
Therefore the product \(\prod_{i} z_i^{\alpha_i}\) is a section of
\(\bundlefromdivisor{- \sum_{i} \alpha_i E_i^*}\)
with divisor supported on end curves.
The product is also an eigen-function.
By~\eqref{eq:3},
the corresponding eigen-space of \(\FiltrationLevel{l'}\) is
\(\zerocohom{\bundlefromdivisor{- l' - l}}\) for the \(l \in Q\)
making the line bundles 
\(\bundlefromdivisor{- \sum_{i} \alpha_i E_i^*}\)
and
\(\bundlefromdivisor{- l' - l}\)
differ by a divisor supported on the exceptional curves.
In other words,
\(l\) is the rational part of \(\sum_{i} \alpha_i E_i^* - l'\)
by the definition \eqref{eq:1} of \(Q\).

Since the divisor of \(\prod_{i} z_i^{\alpha_i}\) in
\(\bundlefromdivisor{- \sum_{i} \alpha_i E_i^*}\) does not contain exceptional curves,
the product is a section of \(\zerocohom{\bundlefromdivisor{- l' - l}}\)
if and only if
\(- \sum_{i} \alpha_i E_i^* \geq - l' - l\)
ie
\(\sum_{i} \alpha_i E_i^* - l' \geq  l\).
Since \(l\) is the rational part of the left-hand side,
the last inequality is equivalent to \(\sum_{i} \alpha_i E_i^* - l' \geq  0\).
\end{proof}

\subsection{The divisorial filtration is generated by monomials}
\label{sec:divis-filtr-gener-mon}

We prove \autoref{th:06} ie that
levels of
the divisorial filtration
are generated by monomials of end curve sections
for singularities
satisfying the \hyperref[def:1]{End Curve Condition}.
First we reformulate it to be suitable for use
in the proof of
the \hyperref[th:3]{End Curve Theorem \ref*{th:3}}.

\begin{proposition}
  \label{prop:1}
  For every resolution of a singularity
  satisfying the \hyperref[def:1]{End Curve Condition},
  and every collection of end curve sections
  belonging to end curves such that
  every exceptional curve is intersected by
  at least two end curves and exceptional curves altogether,
  the following holds.
  Every element of every level of the divisorial filtration is
  a convergent power series in the monomials of the end curve sections
  lying in the same level of the divisorial filtration.
\end{proposition}

This is a simple consequence of
\autoref{lem:4}\ref{item:25}
via standard arguments on complete rings,
as the proof below shows.

\begin{proof}[Proof of \autoref{prop:1}]
Let \(m\) denote the maximal ideal of the function ring
of the universal abelian cover.
The topology of the ring is given by
the \(m\)-adic filtration, ie the powers of \(m\) 
as a neighbourhood basis of \(0\).
As a preliminary,
we show that
the divisorial filtration
gives the same topology and thus the same completion,
ie every ideal of either filtration contains
an ideal from the other filtration.
Clearly, every ideal of the divisorial filtration contains
a large power of the maximal ideal
since all functions in the latter ideal have large vanishing orders
on all the exceptional curves.

We now prove the other direction.
By \cite[(3.2)]{MR713236},
there exists a non-zero effective cycle \(a \in \effectivecycles\)
such that for \(k \geq 2\)
\begin{equation*}
  {\zerocohom{\widetilde{Y}, \bundlefromdivisor[{\widetilde{Y}}]%
      {- \ApplyFunction{{\widetilde{c}}^*}{a}}}}^{\otimes k} \longrightarrow
  \zerocohom{\widetilde{Y}, \bundlefromdivisor[{\widetilde{Y}}]{- k \cdot \ApplyFunction{{\widetilde{c}}^*}{a}}}
\end{equation*}
is onto.
Hence \(\FiltrationLevel{ka} \subseteq {\FiltrationLevel{a}}^k\) lies in
the \(k\)th power of the maximal ideal.

We now turn the proof of the proposition.
Recall that the levels of the divisorial filtration
are isomorphic to
natural line bundles.
Hence applying \autoref{lem:4}\ref{item:25} to
the eigenspaces of the levels show that
that the eigenspaces of factors of two comparable levels
are generated by monomials.
In particular, the factor of any two comparable levels is also
generated by monomials.
Here we call two levels \emph{comparable}
if one is contained in the other,
as usual for partial orders.

We prove that the ring homomorphism is surjective.
By standard completion arguments,
it is enough to prove that the maximal ideal is mapped surjectively
onto \(\smartfrac{m}{m^2}\).
By extension of the constant functions, this is equivalent to
the homomorphism being surjective onto \(\smartfrac{\FiltrationLevel{0}}{m^2}\),
where the level \(\FiltrationLevel{0}\) is the whole function ring.
As we have proved,
there is a level \(\FiltrationLevel{l'}\) of the divisorial filtration
contained in \(m^2\),
hence it is enough to prove that the homomorphism is surjective onto
\(\smartfrac{\FiltrationLevel{0}}{\FiltrationLevel{l'}}\).
This is indeed the case,
as this factor is a factor of
two comparable levels of the divisorial filtration.

Finally, we show that every level \(\FiltrationLevel{l'}\) of
the divisorial filtration as an ideal is generated by monomials,
which obviously shows that
every element of the level is a convergent power series in
the monomials lying in the level.
The level is a finitely generated ideal,
hence by Nakayama's lemma
it is enough to show that \(\smartfrac{\FiltrationLevel{l'}}{m\FiltrationLevel{l'}}\)
is generated by monomials.
We prove this by repeating the argument in the previous paragraph:
there is a level \(\FiltrationLevel{a}\) of the divisorial filtration
contained in \(m\FiltrationLevel{l'}\), and
the factor of two levels \(\smartfrac{\FiltrationLevel{l'}}{\FiltrationLevel{a}}\)
is generated by monomials,
hence so is \(\smartfrac{\FiltrationLevel{l'}}{m\FiltrationLevel{l'}}\).
\end{proof}

\subsection{End Curve Condition implies Monomial Condition}
\label{sec:end-curve-monomial}

In all known proofs that a singularity is splice-quotient,
it is an important step to show that the resolution graph
satisfies the \hyperref[item:21]{Monomial Condition}.
In this subsection, we do this step for the \hyperref[th:3]{End Curve Theorem \ref*{th:3}}.

\begin{proposition}\label{prop:end-to-monom}
  If a good resolution of an isolated normal surface singularity
  with rational homology sphere link
  satisfies the \hyperref[def:1]{End Curve Condition},
  then it also satisfies the \hyperref[item:21]{Monomial Condition}
  for every finite set of end curves
  intersecting the exceptional divisor in pairwise disjoint points
  and every exceptional curve being intersected by at least two
  exceptional curves and end curves altogether.
  %%SPLIT sentence, too long.
\end{proposition}

To better understand the proof, first we present the idea behind it.
First, the Monomial Condition is local, ie it holds for a branch of a vertex
if and only if it holds in the subgraph spanned by the branch and the vertex.
Restricting to this subgraph has the advantage that
the other branches of the vertex become just end curves,
hence we have admissible monomials for them.
The splice diagram equations suggest that
an admissible monomial for the branch is
a linear combination of admissible monomials of two other branches
up to higher degree terms.
So we will take a suitable linear combination of
admissible monomials (ie end curve sections) of two other branches and
cut out the higher degree terms
to obtain an admissible monomial.

This is done precisely as follows.
\begin{proof}[Proof of \autoref{prop:end-to-monom}]
We verify the Monomial Condition for every branch of
every vertex \(v\)
having at least \(3\) branches.

The Monomial Condition for a branch in the whole graph
is equivalent to the Monomial Condition for the branch in
the subgraph spanned by \(v\) and the branch,
since the required equation~\eqref{eq:24}
means the same condition for both graphs.
(The easiest way to see this is that
an equivalent form of \eqref{eq:24} is that
both of its side has the same intersection number with
the exceptional curves of the graph.
The intersection numbers with exceptional curves outside the subgraph is
always \(0\) for both sides,
hence only the subgraph matters.)

Hence we shall work with this subgraph,
and our notations will be implicitly used for this subgraph.

The advantage of using this subgraph is that
\(v\) inherits end curves
cut out by the other branches of \(v\),
see \autoref{lem:1}\ref{item:9}.
Since \(v\) has at least \(3\) branches,
it has at least \(2\) end curves
intersecting it in different points
in the subgraph.

We select two end curve sections \(s_1\) and \(s_2\)
of \(\bundlefromdivisor{- E_v^*}\)
whose end curves intersect \(E_v\)
in different points.
Obviously, these sections lie outside of
\(\zerocohom{\bundlefromdivisor{- E_v^* - E_w}}\)
where \(w\) is the vertex of the branch adjacent to \(v\).

In particular, their images under the embedding
\begin{equation*}
  \factorsectionsfromdivisor{- E_v^*}{E_w}
  \hookrightarrow \zerocohom{\bundlefromdivisor[{E_w}]{- E_v^*}} \simeq \setC
\end{equation*}
are non-zero.
Since the codomain is \(1\) dimensional,
there exists a non-zero complex number \(\alpha\)
such that
\begin{equation}\label{eq:35}
  s_1 - \alpha s_2 \in \zerocohom{\bundlefromdivisor{- E_v^* - E_w}}.
\end{equation}
(This is supposed to be an admissible monomial of the branch plus higher degree terms.)
Since the end curves (ie the divisors) of \(s_1\) and \(s_2\) intersect \(E_v\)
in \emph{different} points,
the linear combination \(s_1 - \alpha s_2\)
is not zero on \(E_v\).
Thus, \(s_1 - \alpha s_2\) is a non-zero element of the factor
\(\factorsectionsfromdivisor{- E_v^* - E_w}{E_v}\).
Hence by \autoref{lem:4}\ref{item:10},
the line bundle
\(\zerocohom{\bundlefromdivisor{- E_v^* - E_w}}\)
has a non-zero section whose divisor is supported by
the end curves and exceptional curves other than \(E_v\).
Since \(\scalarproduct{E_v}{\Chernclass{\bundlefromdivisor{- E_v^* - E_w}}} = 0\),
the divisor cannot contain
any curves intersecting \(E_v\),
so the divisor is supported by the branch and its end curves
(other than \(E_v\)).
In the larger bundle \(\bundlefromdivisor{- E_v^*}\),
the section is still supported on the branch and its end curves,
hence the coefficients of the divisor are
non-negative numbers \(\alpha_u\) for
the end curves and exceptional curves of branch, respectively,
satisfying~\eqref{eq:24}.
\end{proof}

\subsection{Proof of the End Curve Theorem}
\label{sec:proving-end-curve}

In this subsection,
we prove the End Curve Theorem.
Our arguments are taken from \cite[Section~5]{Ouac-c},
and adapted to the use of \autoref{th:06} instead of the
properties of special singularities.

\begin{proof}[Proof of \autoref{th:3}]
We are given a ring homomorphism
from the ring of convergent power series
to the ring of functions of
the universal abelian cover of a singularity
satisfying the \hyperref[def:1]{End Curve Condition}.
The variables are mapped to end curve sections.
What we have to show is that
this homomorphism is surjective with
kernel generated by splice diagram equations.

The surjectivity is part of \autoref{prop:1}, which is already proven.
It follows at once that the kernel is a prime ideal of dimension \(2\).
Since any ideal generated by splice diagram equations is also prime of
dimension \(2\) by \autoref{th:6}, it is enough to show that the kernel
contains a system of splice diagram equations.

We fix a vertex \(v\) with at least \(3\) branches.
We are going to find splice diagram equations \eqref{eq:9} for this vertex
as follows.

First of all, we choose admissible monomials \(M_{v,C}\) for the branches
arbitrarily.
The existence of admissible monomials was shown
in \autoref{prop:end-to-monom}.

By \autoref{lem:2} via \eqref{eq:22},
the monomials having \(v\)-degree at least \(- \scalarproduct{E_v^*}{E_v^*}\)
are exactly the monomials in \(\FiltrationLevel{E_v^*}\).
The admissible monomials lie in
the subspace of
the \(\ApplyFunction{\theta}{\ClassOf{E_v^*}}\)-eigenspace of \(\FiltrationLevel{E_v^*}\),
which is isomorphic to
\(\zerocohom{\bundlefromdivisor{- E_v^*}}\).
Under such an isomorphism,
the vector space of power series with higher \(v\)-degree
and in the same eigenspace of the group action is
the subspace \(\zerocohom{\bundlefromdivisor{- E_v^* - E_v}}\).
So we have to show that
the admissible monomials of \(v\)  satisfy \(\degree{v} - 2\)
linear equations in the factor
\(\factorsectionsfromdivisor{- E_v^*}{E_v}\)
whose matrix of coefficients has all its maximal minors non-degenerate.
The latter condition means that
the subspace generated by all the admissible monomials in the factor
is generated by any two of the admissible monomials.

To see that this condition really holds,
we consider the embedding
\begin{equation}\label{eq:26}
  \factorsectionsfromdivisor{- E_v^*}{E_v} 
  \hookrightarrow \zerocohom{\bundlefromdivisor[{E_v}]{- E_v^*}} \simeq \setC^2.
\end{equation}
The space on the right-hand side has dimension \(2\).
Any non-zero section of it has
exactly one zero point on \(E_v\).
Any two non-zero sections having different zero points
are linearly independent,
and therefore form a basis of it.
Hence any two of the admissible monomials form a basis of it,
thus any two of the admissible monomials generate
the subspace generated by all of the admissible monomials,
as claimed.
\end{proof}

\section{Counterexample for Hilbert\textendash Samuel function}
\label{sec:count-hilb-funct}

In this section
we give an example that the Hilbert\textendash Samuel function
of a splice-quotient singularity is not determined by the resolution graph.
Let \(m\) denote the maximal ideal of an isolated singularity.
Recall that the Hilbert\textendash Samuel function of the singularity is the power series
\begin{equation}
  \label{eq:20}
  \sum_{k=0}^{\infty} \dim \smartfrac{m^k}{m^{k+1}} \cdot t^k \in \setZ[[t]].
\end{equation}

\subsection{The singularities}
\label{sec:singularities}

In this subsection,
we present the example and
state the Hilbert\textendash Samuel functions.
The computation of the functions are left to \autoref{sec:comp-hilb-funct}.

The example is a star-shaped graph:
ie there is a central vertex, and the branches of the vertex are
simply paths.
In our example,
there are four branches whose determinant we denote by
\(a_1\), \(a_2\), \(b\) and \(c\).
We require these four numbers to be pairwisely relative prime.
For every such quadruple of numbers,
there is a unique graph whose determinant is \(1\).
Our example will be this unique graph,
so the universal abelian cover will be the resolution itself
(this is just for simplifying computations).

Now we write up a possible system of splice-quotient equations,
which define the singularity:
\begin{equation}
  \label{eq:5}
  \begin{aligned}
    x_1^{a_1} + y^b +   z^c &= 0, \\
    x_2^{a_2} + y^b + \gamma z^c &= 0,
  \end{aligned}
\end{equation}
where \(\gamma\) is a complex number different from \(0\) and \(1\).
Its Hilbert\textendash Samuel function is
\begin{equation}
  \label{eq:15}
  \smartfrac{{( 1 - t^{a_1} )} {( 1 - t^{a_2} )}}{{( 1 - t )}^{4}}
\end{equation}

By adding higher degree terms to the equations,
we obtain another singularity with the same resolution graph:
\begin{equation}
  \begin{aligned}
    \label{eq:7}
    x_1^{a_1} + y^b +   z^c                   & = 0, \\
    x_2^{a_2} + y^b + \gamma z^c + x_1^{a_1 - 1} x_2^i & = 0.
  \end{aligned}
\end{equation}
Here we impose the following restrictions on the numbers:
\begin{subequations}
  \begin{align}
    \label{eq:10}
    2 \leq a_1 &< a_2 < b < c, \\
    \label{eq:12}
    i a_1 &> a_2, \\
    \label{eq:11}
    a_1 - 1 + i &< a_2, \\
    \label{eq:13}
    b + a_1 - 1 &> 2 a_2 - i.
  \end{align}
\end{subequations}
The inequality of \eqref{eq:12} expresses that
the added term is of higher degree.
The other inequalities are of technical nature.
In particular,
the sole purpose of \eqref{eq:13} is to simplify computations.

The Hilbert\textendash Samuel function of the latter singularity is
\begin{equation}
  \label{eq:16}
  \frac{%
    \frac{t - t^{a_1 - 1}}{1 - t}
    \cdot
    \frac{1 - t^{a_2}}{1 - t}
    +
    \frac{1 - t^{2 a_2 - i}}{1 - t}
    + t^{a_1 - 1}
    \frac{1 - t^{i}}{1 - t}
  }{{( 1 - t )}^{2}}
\end{equation}

\subsection{Computation of Hilbert\textendash Samuel function}
\label{sec:comp-hilb-funct}

In this subsection,
we verify the Hilbert\textendash Samuel functions~\eqref{eq:15} and \eqref{eq:16}.

Recall that the \emph{codegree} of a power series is
the minimum of the degrees of its terms.
Obviously,
a function is in the \(k\)th power of the maximal ideal
if and only if
it is represented by a power series of codegree at least \(k\).

The idea is to find a good basis for the function rings:
ie a collection of monomials such that every function can be uniquely written as
a convergent power series having only these monomials as its terms.
Furthermore,
these unique representation should have the minimal possible codegree
so that it reflects in which power of maximal ideal the function is contained.

Given such a collection of monomials,
the monomials of degree \(k\) obviously form
a basis of the factor \(m^k / m^{k+1}\),
leading to a combinatorial formula for the Hilbert\textendash Samuel function.

We claim that for the equations \eqref{eq:5},
a good collection of monomials are the ones not divisible by
\(x_1^{a_1}\) or \(x_2^{a_2}\).
For the equations \eqref{eq:7},
a good collection is formed by the monomials divisible by none of
\(x_1 x_2^{a_2}\),
\(x_1^{a_1}\),
\(x_1^{a_1 - 1} x_2^{i}\) and
\(x_2^{2 a_2 - i}\).

The claimed form of the Hilbert\textendash Samuel function
follows easily from these claims.

We verify our claim only for the equations \eqref{eq:7},
since the verification for \eqref{eq:5} is similar and easier.
We call the monomials outside of the good collection to be \emph{forbidden}.

The verification is based on expressing the forbidden monomials in terms of
the allowed ones:
\begin{subequations}\label{eq:23}
  \begin{align}
    \label{eq:6}
    x_1 x_2^{a_2} &=
    \begin{multlined}[t]
      (x_2^i - x_1) y^b + (x_2^i - \gamma x_1) z^c \\
      - x_2^i (x_1^{a_1} + y^b + z^c)
      + x_1 (x_2^{a_2} + y^b + \gamma z^c + x_1^{a_1 - 1} x_2^i)
    \end{multlined}
    \\
    \label{eq:8}
    x_1^{a_1} &=
    \begin{multlined}[t]
      - y^b - z^c \\
      + (x_1^{a_1} + y^b + z^c)
    \end{multlined}
    \\
    \label{eq:17}
    x_1^{a_1 - 1} x_2^{i} &=
    \begin{multlined}[t]
      - x_2^{a_2} - y^b - \gamma z^c \\
      + (x_2^{a_2} + y^b + \gamma z^c + x_1^{a_1 - 1} x_2^i)
    \end{multlined}
    \\
    \label{eq:18}
    x_2^{2 a_2 - i} &=
    \begin{multlined}[t]
      (x_1^{a_1 - 1} - x_2^{a_2 - i} - x_1^{a_1 - 2} x_2^i) y^b +
      (\gamma x_1^{a_1 - 1} - \gamma x_2^{a_2 - i} - x_1^{a_1 - 2} x_2^i) z^c \\
      + x_1^{a_1 - 2} x_2^i (x_1^{a_1} + y^b + z^c) \\
      + (x_2^{a_2 - i} - x_1^{a_1 - 1}) (x_2^{a_2} + y^b + \gamma z^c + x_1^{a_1 - 1} x_2^i)
    \end{multlined}
  \end{align}
\end{subequations}
In these equations,
the first line contains the equivalent expression without
the forbidden monomials, and of non-smaller codegree
thanks to \eqref{eq:11} and \eqref{eq:13}.
Further lines contain a linear combination of the defining equations
to make the equation hold, which is just for the reader's convenience.

We now verify that every function is represented by a unique power series
not containing a term divisible by the mentioned monomials.
By the Weierstrass Preparation Theorem,
the ring
\(\setC\{x_1, x_2, y, z\} / (x_2^{a_2} + y^b + \gamma z^c + x_1^{a_1 - 1} x_2^i)\)
is a free module over \(\setC \{x_1, y, z\}\) with basis
\(1\), \(x_2\),\dots, \(x_2^{a_2 - 1}\).
Similarly,
the ring
\(\setC \{x_1, y, z\} / (x_1^{a_1} + y^b + z^c)\)
is a free module over \(\setC \{ y, z\}\) with basis
\(1\), \(x_1\),\dots, \(x_1^{a_1 - 1}\).
It follows that
\(\setC\{x_1, x_2, y, z\} / (x_1^{a_1} + y^b + z^c,
                       x_2^{a_2} + y^b + \gamma z^c + x_1^{a_1 - 1} x_2^i)\)
is a free module over \(\setC \{ y, z\}\) with basis
\(x_1^{j_1} x_2^{j_2}\) for \(0 \leq j_1 \leq a_1 - 1\) and \(0 \leq j_2 \leq a_2 - 1\).
In this basis,
the monomials \(x_1^{a_1 - 1} x_2^{j}\) for \(i \leq j \leq a_2 - 1\)
can be replaced by \(x_2^{a_2 + j -i}\) by~\eqref{eq:17}.
This gives us the claimed basis of the function ring.

We still need to check that the unique representative is of minimal codegree.
We do this by transforming any representative \(f_0\) into the unique one omitting
the forbidden monomials such that the codegree does not decrease.
The main idea behind this is to replace forbidden monomials with another
expressions
of non-smaller codegree.
Equations \eqref{eq:23} provide these replacements.

In more details,
we first eliminate the monomials divisible by \(x_2^{2 a_2 - i}\) via
\eqref{eq:18} and the Weierstrass Preparation Theorem,
which does not decrease the codegree.
Then we eliminate the terms divisible by \(x_1 x_2^{a_2}\) via \eqref{eq:6}.
The third step is to eliminate the terms divisible by \(x_1^{a_1}\) via
\eqref{eq:8} and the Weierstrass Preparation Theorem.
The final step is the elimination of monomials
divisible by \(x_1^{a_1 - 1} x_2^{i}\) via\eqref{eq:17}.
The reader can easily verify that during these steps
forbidden monomials claimed to be
eliminated in previous steps
cannot reappear,
so finally we have obtained a representative without forbidden monomials
and with codegree at least the original.
\bibliographystyle{amsplainurl}\bibliography{hivatkozas}
\end{document}